\documentclass[reqno,12pt]{amsart}
\usepackage{amssymb,amsmath,amsthm,amsxtra,mathtools,mathrsfs,calc,bm}
\usepackage{enumerate}
\usepackage[margin=1in]{geometry}
\usepackage{calc}
\usepackage{graphicx}
\usepackage{caption}
\usepackage{subcaption}

\usepackage{mleftright}
\mleftright
\usepackage{nccmath}
\def\Z{\mathbb{Z}}
\def\Q{\mathbb{Q}}
\def\R{\mathbb{R}}
\def\H{\mathbb{H}}
\def\N{\mathbb{N}}
\def\C{\mathbb{C}}

\def\D{\mathbb{D}}
\def\sQ{\mathcal{Q}}

\DeclareMathOperator{\im}{Im}
\DeclareMathOperator{\re}{Re}
\DeclareMathOperator{\spt}{spt}
\DeclareMathOperator{\spn}{span}
\def\Res{\operatorname*{Res}}
\def\Spin{\operatorname*{Spin}}
\DeclareMathOperator{\Det}{Det}

\DeclareMathOperator{\Tr}{Tr}
\DeclareMathOperator{\Mp}{Mp}
\DeclareMathOperator{\sgn}{sgn}
\def\SL{{\rm SL}}
\def\GL{{\rm GL}}

\newcommand{\pfrac}[2]{\left(\frac{#1}{#2}\right)}
\newcommand{\ptfrac}[2]{\left(\tfrac{#1}{#2}\right)}
\newcommand{\pMatrix}[4]{\left(\begin{matrix}#1 & #2 \\ #3 & #4\end{matrix}\right)}
\renewcommand{\pmatrix}[4]{\left(\begin{smallmatrix}#1 & #2 \\ #3 & #4\end{smallmatrix}\right)}
\renewcommand{\bar}[1]{\overline{#1}}

\newcommand{\leg}[2] {\left(\frac{#1}{#2}\right)}
\newcommand{\tleg}[2] {\left(\tfrac{#1}{#2}\right)}
\renewcommand{\(}{\left(}
\renewcommand{\)}{\right)}

\renewcommand{\sl}{\big|}

\def\ep{\epsilon}

\newtheorem{theorem}{Theorem}
\newtheorem{lemma}[theorem]{Lemma}

\newtheorem{proposition}[theorem]{Proposition}
\theoremstyle{remark}
\newtheorem*{remark}{Remark}

\newtheorem*{example}{Example}
\numberwithin{equation}{section}

\begin{document}

\title[Formulas for the smallest parts function]{Algebraic and transcendental formulas for  the smallest parts function}

\date{\today}

\author{Scott Ahlgren}
\address{Department of Mathematics\\
University of Illinois\\
Urbana, IL 61801} 
\email{sahlgren@illinois.edu} 

\author{Nickolas Andersen}
\address{Department of Mathematics\\
University of Illinois\\
Urbana, IL 61801} 
\email{nandrsn4@illinois.edu}
\thanks{The first  author was supported by a grant from the Simons Foundation (\#208525 to Scott Ahlgren).}
\begin{abstract}  Building on work of Hardy and Ramanujan, Rademacher proved a well-known  formula 
for the values of the ordinary partition function $p(n)$.
More recently, Bruinier and Ono obtained an algebraic formula for these values.  Here we study the smallest parts function introduced by Andrews;
$\spt(n)$  counts the number of smallest parts in the partitions of $n$.  The generating function for $\spt(n)$ forms a component of a natural mock modular 
form of weight $3/2$ whose shadow is the Dedekind eta function.
Using automorphic methods (in particular the theta lift of Bruinier and Funke),
we obtain an exact formula and an algebraic formula for its values.
In contrast with the case of $p(n)$, the convergence of our expression is non-trivial, and requires power savings estimates 
for weighted sums of Kloosterman sums for a multiplier in weight $1/2$.  These are proved with spectral methods (following an argument of Goldfeld and Sarnak).
\end{abstract}

\maketitle

\allowdisplaybreaks

\section{Introduction}

Let $p(n)$ denote the ordinary partition function.  
Hardy and Ramanujan \cite{HR:asymptotic} developed the circle method to prove the asymptotic formula
\[
	p(n) \sim \frac{e^{\pi\sqrt{\frac{2n}3}}}{4\sqrt{3}\, n}.
\]
Building on their work, Rademacher \cite{Rademacher:partition,  Rademacher:expository, Rademacher:series} proved the famous  formula
\begin{equation} \label{eq:rademacher}
	p(n) = \frac{2\pi}{(24n-1)^{3/4}} \sum_{c=1}^\infty \frac{A_c(n)}{c} I_{\frac32}\pfrac{\pi\sqrt{24n-1}}{6c},
\end{equation}
where $I_{\nu}$ is the $I$-Bessel function, $A_c(n)$ is the Kloosterman sum
\begin{equation}\label{Acn}
A_c(n):=\sum_{\substack{d \bmod c\\ (d,c)=1}}e^{\pi i s(d,c)}e\(-\frac{dn}c\), \quad e(x):=e^{2\pi i x}
\end{equation}
and $s(d,c)$ is the Dedekind sum
\begin{equation} \label{eq:def-dedekind-sum}
	s(d,c) := \sum_{r=1}^{c-1} \frac rc \left( \frac {dr}c - \left\lfloor \frac {dr}c \right\rfloor - \frac 12 \right).
\end{equation}
The existence of formula \eqref{eq:rademacher} is made possible by the fact that the generating function for $p(n)$ is a modular form of weight $-1/2$, namely
\[
	q^{-\frac1{24}} \sum_{n\geq 0} p(n) q^n = \frac{1}{\eta(\tau)}, \quad q:=\exp(2\pi i \tau),
\]
where $\eta(\tau)$ denotes the Dedekind eta function.

There are a number of ways to prove \eqref{eq:rademacher}.  For example, Pribitkin \cite{Pribitkin:partition} obtained a proof using a modified Poincar\'e  series which represents $\eta^{-1}(\tau)$.
A similar technique can be used to obtain general formulas for the coefficients of modular forms of negative weight (see e.g. Hejhal \cite[Appendix D]{Hejhal} or Zuckerman \cite{zuckerman}).
The authors \cite{AA:weakharmonic} recovered \eqref{eq:rademacher} from Poincar\'e series representing a  weight $5/2$ harmonic Maass 
form whose shadow is $\eta^{-1}(\tau)$.
As pointed out by Bruinier and Ono \cite{BO:algebraic}, the exact formula can  be recovered from the algebraic formula 
\eqref{algpart} stated below
(this was partially carried out by Dewar and Murty \cite{DewarMurty}).
The equivalence of \eqref{eq:rademacher}  and \eqref{algpart}   (in a more general setting) is made 
explicit by \cite[Proposition~7]{Andersen:singular}.

The smallest parts function $\spt(n)$,  introduced by  Andrews  in \cite{Andrews:spt}, counts the number of smallest parts in the partitions of $n$.
Andrews proved that the generating function for $\spt(n)$ is given by
\[
	S(\tau) := \sum_{n\geq 1} \spt(n) q^n = 
\prod_{n \geq 1} \frac{1}{1-q^n}\left(\sum_{n \geq 1} \frac{nq^n}{1-q^n} + \sum_{n \neq 0} \frac{(-1)^nq^{n(3n+1)/2}}{(1-q^n)^2}\right).
\]
Work of Bringmann \cite{Bringmann:duke} shows that $S(\tau)$ is a component of a mock modular form of weight $3/2$ whose shadow is the Dedekind eta-function; using the circle method, she obtained an asymptotic expansion for $\spt(n)$.  In particular we have 
\[
	\spt(n) \sim \frac{\sqrt{6n}}{\pi} \, p(n).
\]
Many authors have investigated the  coefficients of this mock modular form, which is a prototype for 
modular forms of this type (see, e.g., 
\cite{ABL:spt,AhlgrenKim:spt,AGL:spt,FO:spt,Garvan:sptcong,Garvan:sptcong2,GarvanJennings:spt,Ono:spt}).

In analogy with \eqref{eq:rademacher}, we  prove the following   formula for $\spt(n)$.
\begin{theorem} \label{thm:spt-exact} 
For all $n\geq 1$, we have
\begin{equation} \label{eq:spt-exact}
	\spt(n) = \frac\pi6 (24n-1)^{\frac14} \sum_{c=1}^\infty \frac{A_c(n)}{c}\left(I_{1/2}-I_{3/2}\right)\left(\frac{\pi\sqrt{24n-1}}{6c}\right).
\end{equation}
\end{theorem}
The convergence of Rademacher's formula for $p(n)$ follows from elementary estimates on the size of the $I$-Bessel function and the trivial bound $|A_c(n)|\leq c$.
By contrast, the convergence of the series for $\spt(n)$ is quite subtle, and requires non-trivial estimates for weighted sums of the Kloosterman sum \eqref{Acn}.
For this we utilize the spectral theory of half-integral weight Maass forms, discussed in more detail below.
 
We mention that Bringmann and Ono \cite{BO:fqconj}  established an exact formula for the coefficients of the weight $1/2$ mock theta function $f(q)$; this proved  conjectures of  Dragonette \cite{Dragonette:mocktheta} and Andrews \cite{Andrews:mocktheta}. In this work they required estimates for Kloosterman sums of level $2$,
which were obtained by adapting a method of Hooley.

Formulas \eqref{eq:rademacher} and \eqref{eq:spt-exact} express integers as infinite series involving values of transcendental functions.
By contrast, Bruinier and Ono \cite{BO:algebraic} (see also \cite{BO:partition}) obtained a formula for $p(n)$ as a finite sum of algebraic numbers.
Let $P(\tau)$ denote the $\Gamma_0(6)$-invariant function 
\begin{equation} \label{eq:def-P}
	P(\tau) := -\frac12\left(q\frac{d}{dq} + \frac{1}{2\pi y}\right) \frac{E_2(\tau)-2E_2(2\tau)-3E_2(3\tau)+6E_2(6\tau)}{(\eta(\tau)\eta(2\tau)\eta(3\tau)\eta(6\tau))^2}.
\end{equation}
For $n\geq 1$ define
\[
	\sQ_{1-24n}^{(1)} := \left\{ ax^2+bxy+cy^2: b^2-4ac=1-24n, \, 6\mid a>0, \, \text{ and } b\equiv 1\bmod 12 \right\}.
\]
The group 
\begin{equation}\label{eq:gamma_def}
\Gamma := \Gamma_0(6)/\{\pm 1\}
\end{equation}
 acts on this set.
For each $Q\in \sQ_{1-24n}^{(1)}$ let $\tau_Q$ denote the root of $Q(\tau,1)$ in the upper-half plane $\H$.
Bruinier and Ono showed that
\begin{equation}\label{algpart}
	p(n) = \frac{1}{24n-1} \sum_{Q\in \Gamma\backslash \sQ_{1-24n}^{(1)}} P(\tau_Q).
\end{equation}

We obtain an analogue of \eqref{algpart} for  $\spt(n)$.
Define the weakly holomorphic modular function
\begin{equation} \label{eq:def-f}
	f(\tau) := \frac{1}{24}\frac{E_4(\tau) - 4E_4(2\tau) - 9E_4(3\tau) + 36E_4(6\tau)}{(\eta(\tau)\eta(2\tau)\eta(3\tau)\eta(6\tau))^2}.
\end{equation}
Then we have the following algebraic formula.
\begin{theorem} \label{thm:spt-trace} For all $n\geq 1$, we have
\begin{equation} \label{eq:spt-trace}
	\spt(n) = \frac1{12}\sum_{Q\in \Gamma\backslash \sQ_{1-24n}^{(1)}} \left(f(\tau_Q) - P(\tau_Q)\right).
\end{equation}
\end{theorem}
Bruinier and Ono showed that the values $P(\tau_Q)$ are algebraic numbers with bounded denominators, and the classical theory of complex multiplication implies that the values $f(\tau_Q)$ are algebraic as well (see, for instance, Section 6.1 of \cite{Zagier:modular}).

The proof of Theorem \ref{thm:spt-trace} relies on the theta lift of Bruinier and Funke \cite{BF:traces} which relates coefficients of harmonic Maass forms of weight $3/2$ to traces of CM values of modular functions.
Generalizations of this lift are given in \cite{alfes,alfes-ehlen,BO:algebraic}.

\begin{example}
We illustrate the simplest case of Theorem \ref{thm:spt-trace}. The class number of $\Q(\sqrt{-23})$ is $3$, so
 $\Gamma \backslash \sQ_{-23}^{(1)}$ consists of $3$ classes.  These are represented by the forms
\[
	Q_1 = 6x^2+xy+y^2, \quad Q_2 = 12x^2+13xy+4y^2, \quad Q_3 = 18x^2+25xy+9y^2,
\]
whose roots are
\[
	\tau_1 = \frac{-1+\sqrt{-23}}{12}, \quad \tau_2 = \frac{-13+\sqrt{-23}}{24}, \quad \tau_3 = \frac{-25+\sqrt{-23}}{36}.
\]
Let $g=f-P$. Since the values $\{g(\tau_k)\}$ are conjugate algebraic numbers, we find that
\begin{equation}\label{classpoly}
	\prod_{k=1,2,3}(x-g(\tau_k)) = x^3 - 12x^2 - \frac{1008}{23}x - \frac{1728}{23}
\end{equation}
by approximating each $g(\tau_k)$ using \eqref{eq:def-P} and \eqref{eq:def-f}.
This shows that $\spt(1)=1$.
Computing the roots of the polynomial in \eqref{classpoly} gives the values
\[
g(\tau_1) = 4\left(1+\frac{2}{23}\beta+\frac{22}{\beta}\right), \quad
g(\tau_2) = 4\left(1+\frac{2}{23}\zeta_3\beta+\frac{22\zeta_3^2}{\beta}\right), \quad
g(\tau_3) = \bar{g(\tau_2)},
\]
where $\zeta_3 := e^{2\pi i/3}$ and
\[
	\beta := \sqrt[3]{\frac{23}{2}\(391+21\sqrt{69}\)}.
\]
\end{example}

We return to the problem of obtaining estimates for weighted sums of the $A_c(n)$, which is 
of independent interest.  Define
\[
	{\bf A}_n(x) := \sum_{c\leq x} \frac{A_c(n)}{c}.
\]
Lehmer \cite[Theorem 8]{Lehmer:series} proved the sharp Weil-type bound 
\begin{equation} \label{eq:super_weil} 
	|A_c(n)|<2^{\omega_o(c)} \sqrt c,
\end{equation}
where $\omega_o(c)$ is the number of distinct odd primes dividing $c$.
Rademacher \cite{rademacher-indian}  later simplified Lehmer's treatment of the sums $A_c(n)$ using 
Selberg's formula \cite{Whiteman}
\[
	A_c(n) = \sqrt{\frac c3} \sum_{\substack{\ell\bmod {2c} \\ (3\ell^2+\ell)/2 \equiv -n(c)}} (-1)^\ell \cos\left(\frac{6\ell+1}{6c}\pi\right).
\]
From \eqref{eq:super_weil} we obtain
\[{\bf A}_n(x) \ll_{\epsilon} x^{\frac 12+\epsilon},\]
Since this   is not sufficient to prove 
 the convergence of the series in \eqref{eq:spt-exact}, we require a power savings estimate for ${\bf A}_n(x)$.
In Section \ref{kloostersection}  we adapt the method of Goldfeld and Sarnak \cite{GS:kloosterman} to relate ${\bf A}_n(x)$ to the spectrum of the weight $1/2$ hyperbolic Laplacian on $\SL_2(\Z)\backslash\H$.
We prove a result which has the following     corollary.
\begin{theorem} \label{Athm}
Suppose that $-n=\frac{k(3k\pm 1)}2$ is  a pentagonal number.  Then for any $\epsilon>0$ we have
\[
	\sum_{c\leq x} \frac{A_c(n)}{c} = (-1)^k \, \mfrac{12\sqrt{3}}{\pi^2}\,  x^{\frac12}+O\(x^{\frac16+\epsilon}\).
\]
If $-n$ is not  pentagonal  then we have 
\[
	\sum_{c\leq x} \frac{A_c(n)}{c} = O\(x^{\frac16+\epsilon}\).
\]
The implied constants depend on $n$ and $\epsilon$.
\end{theorem}
This  can be compared with Kuznetsov's bound \cite{Kuznetsov} 
\[
	\sum_{c\leq x}\frac{k(m,n;c)}{c}\ll_{m,n}x^\frac16(\log x)^\frac13
\]
where 
$k(m,n;c)$ is the ordinary Kloosterman sum defined in \eqref{eq:ordinary-kloo} below.

\section{Preliminaries}
We briefly introduce some of the objects which we will require. 
\subsection{Quadratic forms and Atkin-Lehner involutions}
Let $M_0^!(\Gamma_0(N))$ denote the space of modular functions on $\Gamma_0(N)$ whose poles are supported at the cusps.
We will mainly work with $N=6$; in this case, there are four cusps, one corresponding to each divisor of $6$.
To move among the cusps, the Atkin-Lehner involutions \cite{AtkinLehner} are useful. 
For each divisor $d$ of $6$, we define the Atkin-Lehner involution $W_d$ on $M_0^!(\Gamma_0(6))$ as the map $f\mapsto f \sl_0 W_d$, where
\begin{equation} \label{eq:def-W-d}
	W_1 = \pMatrix 1001, \ W_2 = \frac{1}{\sqrt2}\pMatrix 2{-1}6{-2}, \ W_3 = \frac{1}{\sqrt3}\pMatrix 3163, \ W_6 = \frac{1}{\sqrt{6}} \pMatrix0{-1}60,
\end{equation}
and $(f\sl_0 \gamma)(\tau) := f(\gamma \tau)$.
The normalizing factors are chosen so that $W_d\in\SL_2(\R)$, which will be convenient later.
If $d,d'\mid 6$, then
\[
	W_d  W_{d'} = W_{\frac{dd'}{(d,d')^2}}.
\]

Suppose that $r\in\{1,5,7,11\}$ and that $D>0$, and
define
\[
	\sQ_{-D}^{(r)} := \left\{ ax^2+bxy+cy^2: b^2-4ac=-D, \ 6\mid a>0, \text{ and } \ b\equiv r\bmod{12} \right\}.
\] 
Let $\Gamma_0^*(6)\subset \SL_2(\R)$ denote the group generated by $\Gamma_0(6)$ and the Atkin-Lehner involutions $W_d$ for $d\mid 6$.
Matrices $g=\pmatrix \alpha\beta\gamma\delta\in\Gamma_0^*(6)$ act on binary quadratic forms on the left by
\begin{equation} \label{eq:Q-action}
	g Q(x,y) := Q(\delta x-\beta y,-\gamma x+\alpha y).
\end{equation}
This action is compatible with the action $g\tau:=\frac{\alpha\tau+\beta}{\gamma\tau+\delta}$ on the root $\tau_Q\in \H$ of $Q(\tau,1)$:
for   $g\in \Gamma_0^*(6)$, we have
\begin{equation} \label{eq:root-compatible}
	g\,\tau_Q = \tau_{g Q}.
\end{equation}
Define
\[
	\sQ_{-D} := \bigcup_{r\in \{1,5,7,11\}} \sQ_{-D}^{(r)}.
\]
A computation involving \eqref{eq:def-W-d} and \eqref{eq:Q-action} shows that
\begin{equation} \label{eq:W-d-r-r'}
	W_d : \sQ_{-D}^{(r)} \longleftrightarrow \sQ_{-D}^{(r')}
\end{equation}
is a bijection, where
\begin{equation} \label{eq:r'-r-cases}
r'\equiv (2d\mu(d)-1)r\pmod{12}.
\end{equation}
For each $r$,  we have
\begin{equation} \label{eq:Qn-decomp-W-d}
	\sQ_{-D} = \bigcup_{d|6} W_d \, \sQ_{-D}^{(r)}.
\end{equation}

\subsection{Quadratic spaces of signature $(1, 2)$} 
The proof of Theorem~\ref{thm:spt-trace} uses a theta lift of Bruinier-Funke associated to  an isotropic rational quadratic space of 
signature $(1,2)$.  To access the necessary results requires some background, which we develop briefly in the next two subsections.
For further details, see \cite{BF:traces, Bruinier:Hilbert}.

Let $V$ be an isotropic rational quadratic space of signature $(1, 2)$
with non-degenerate symmetric bilinear form $(\cdot ,\cdot )$. Let the positive square-free integer $d$ denote the discriminant
of the  quadratic form $q$ given by $q(v)=\frac12(v,v)$. We may view $V$ as the subspace of pure quaternions
with oriented basis
\[\left\{\pMatrix01{-1}0,\   \pMatrix100{-1},\    \pMatrix0110\right\}\]
in the quaternion algebra $M_2(\Q)$;
in other words we have
\[V=\left\{X=\pMatrix {x_1}{x_2}{x_3}{-x_1}\ : \ x_i\in \Q\right\}.\]
With this identification we have
\[q(X)=d\Det(X),\ \ \ (X, Y)=-d\Tr(XY).\]

We identify $G:=\Spin(V)$ with $\SL_2(\Q)$ and $\overline G\simeq \operatorname{PSL}_2(\Q)$ with its image in $\operatorname{SO}(V)$.
The group $G$ acts on $V$ by conjugation; we write
\[g.X:=gXg^{-1}.\]
The group $G(\R)$ acts transitively on the 
 Grassmannian $\D$ of positive lines in $V$:
\[\D:=\left\{z\subseteq V(\R)\ : \ \dim z=1 \ \ \text{and}\ \  q|_z>0\right\}.\]
Choosing  the base point $z_0=\spn\(\pmatrix01{-1}0\)\in \D$, we find that $z_0$ is stabilized 
by $\operatorname{SO}_2(\R)$, so that 
\[\D\simeq \operatorname{SO}_2(\R)\backslash G(\R)\]
is a Hermitian symmetric space.

An explicit isomorphism $\H\simeq \D$  can be described as follows.  For $\tau=x+i y\in \H$, 
 let 
\[g_\tau:=\frac1{\sqrt{y}}\pMatrix{x }{- y}{1}0\in G(\R),\]
so that $g_\tau i=\tau$.
Then define 
\begin{equation} \label{eq:H-D-iso}
	X(\tau):= g_\tau.\pMatrix01{-1}0=\frac1{y}\pMatrix{- x}{x^2+y^2}{-1}{ x},
\end{equation}
and define an isomorphism $\H\rightarrow \D$ by 
$\tau\mapsto\spn\(X(\tau)\)$.
We have
\begin{equation} \label{eq:X-g}
	X(g\tau) = g.X(\tau)
\end{equation}
for all $g\in \SL_2(\R)$.

Let $L\subseteq V(\Q)$ be the lattice
\begin{equation}\label{Ldef}
L:=\left\{\pMatrix b{c/6}{-a}{-b}\ : \ a,b,c\in \Z \right\}. 
\end{equation}
The dual lattice is 
\[L'=\left\{\pMatrix {b/12}{c/6}{-a}{-b/12}\ : \ a,b,c\in \Z \right\},\]
and we  identify $L'/L$ with $\Z/12\Z$.

The group $\Gamma_0(6)\subseteq \Spin(L)$ fixes $L$. 
Let 
\[\{\mathfrak e_h : h\in \Z/12\Z\}\]
 denote the standard basis of the group ring $\C[L'/L]$.
A computation shows that matrices $g=\pmatrix abcd \in \Gamma_0^*(6)$ act on $\C[L'/L]$ by
\begin{equation}\label{frickeaction}
	g.  \mathfrak e_h = \mathfrak e_{(1+2bc)h}.
\end{equation}
In particular, if $g\in \Gamma_0(6)$, then $g$ acts trivially on $L'/L$.

Let 
\[M:=\Gamma_0(6)\backslash \D\]
 be the modular curve.
If $X\in V(\Q)$ has  positive length,  then we define
\[D_X:= \spn(X)\subseteq \D.\]
For each positive rational number $m$ and each $h\in L'/L$, define
\[L_{h,m} := \left\{X\in L+h\  :\  q(X)=m\right\}.
\]
Then $\Gamma_0(6)$ acts on $L_{h,m}$ with finitely many orbits.

The set $\operatorname{Iso}(V)$ of isotropic lines in $V(\Q)$ is identified with the the cusps of $G(\Q)$ via the map
\[(\alpha:\beta)\mapsto \spn\(\pMatrix {-\alpha \beta}{\alpha^2}{-\beta^2}{\alpha\beta}\).\]
The cusps of $M$ are the $\Gamma_0(6)$ classes of $\operatorname{Iso}(V)$;
these are represented by the lines $\ell_j:=\spn(X_j)$, where
\[X_0:=\pMatrix0100,\ \ X_1:=\pMatrix00{-1}0, \ \ X_2:=\pMatrix{-2}{1}{-4}{2}, \ \text{and}\ \ X_3:=\pMatrix{-3}{1}{-9}{3}.\]
For each $\ell\in \operatorname{Iso}(V)$ choose $\sigma_\ell\in \SL_2(\Z)$ with 
$\sigma_\ell \ell_0=\ell$, and let $\alpha_\ell$ be the width of the cusp $\ell$.
For each $\ell$, there is a positive rational number $\beta_\ell$ such that
\[
	\ell_0 \cap \sigma_\ell^{-1} L = \pMatrix{0}{\beta_\ell\Z}{0}{0},
\]
and we define $\varepsilon_\ell := \alpha_\ell / \beta_\ell$.

Suppose that $q(X)<0$ and that $Q(X)\in -6(\Q^\times)^2$.
Then (see \cite[Lemma 3.6]{Funke:Heegner}) $X$ is orthogonal to two isotropic lines, $\spn(Y)$ and $\spn(\widetilde Y)$.
 We associate $\ell_X:=\spn(Y)$  to $X$ if $(X, Y, \widetilde Y)$ is a positively
oriented basis of $V$. We then have $\ell_{-X}=\spn(\widetilde Y)$.
For each $\ell$, define
\[
	L_{h,-6m^2,\ell} := \left\{ X\in L_{h,-6m^2} : \ell_X = \ell \right\}.
\]
Then $\Gamma_0(6)$ acts on these sets, and equation (4.7) of \cite{BF:traces} shows that
\begin{equation} \label{eq:def-v-ell}
	v_\ell(h,-6m^2) := \left| \Gamma_0(6)\backslash L_{h,-6m^2,\ell} \right| = 
	\begin{cases}
		2m\varepsilon_\ell & \text{ if } L_{h,-6m^2,\ell} \neq \emptyset, \\
		0 & \text{ otherwise.}
	\end{cases}
\end{equation}

\subsection{Harmonic Maass forms and the theta lift}\label{thetaliftsec}

Let $\Mp_2(\mathbb{R})$ denote the metaplectic two-fold cover of $\SL_2(\mathbb{R})$. The elements of this group are pairs $(M,\phi(\tau))$, where
\[
	M=\pmatrix abcd \in \SL_2(\mathbb{R}),
\]
and $\phi:\H\to \C$ is a holomorphic function satisfying $\phi(\tau)^2=(c\tau+d)$.

Let $\Mp_2(\Z)$ denote the inverse image of $\SL_2(\mathbb{Z})$ under the covering map; this group is generated by
\[
	T = \left(\pMatrix 1101, 1\right) \quad \text{ and } \quad S = \left(\pMatrix0{-1}10, \sqrt{\tau}\right)
\]
(here and throughout, we take the principal branch of $\sqrt{\cdot}$\,).
We fix the lattice $L$ defined by \eqref{Ldef}.
The   Weil representation (see Chapter 1 of \cite{Bruinier:borcherds})
\[
	\rho_L : \Mp_2(\mathbb{Z}) \to \GL\left(\mathbb{C}[L'/L]\right)
\]
 is defined by
\begin{equation}\label{weilrep}
\begin{aligned}
	\rho_L(T) \mathfrak{e}_h &= e\left(-\tfrac{h^2}{24}\right) \mathfrak{e}_h, \\
	\rho_L(S) \mathfrak{e}_h &= \frac{\sqrt{i}}{\sqrt{12}} \sum_{h' \in L'/L} e\(\frac{hh'}{12}\) \mathfrak{e}_{h'}.
\end{aligned}
\end{equation}

Denote by $H_{k, \rho_L}$ the space of weak harmonic Maass forms of weight $k$ for the representation $\rho_L$; these are functions
$F:\mathbb{H} \to \mathbb{C}[L'/L]$ which satisfy the following conditions:
\begin{enumerate}
\item For  $(\gamma,\phi)\in \Mp_2(\mathbb{Z})$,
\[
f(\gamma \tau) = \phi(\tau)^{2k} \, \rho_L(\gamma,\phi) \, f(\tau).
\]
\item $\Delta_k f =0$ , where 
\[
\Delta_k := -y^2 \left( \frac{\partial^2}{\partial x^2} + \frac{\partial^2}{\partial y^2} \right) +   i k y \left( \frac{\partial}{\partial x} + i \frac{\partial}{\partial y}  \right).
\]
\item $f$ has at most linear exponential growth at $\infty$.
\item $\xi_k f$ is holomorphic at $\infty$, where
\[\xi_k := 2iy^k \bar{\frac{\partial }{\partial \bar{\tau}}}.\]
\end{enumerate}

The theta lift of $f\in M_0^!(\Gamma_0(6))$ is given by 
\begin{equation}\label{thetalift}
I(\tau, f):=\int_Mf(z)\Theta(\tau, z)=\sum_{h\in L'/L}I_h(\tau, f) {\mathfrak e}_h,
\end{equation}
with
\begin{equation} \label{thetalift-h}
I_h(\tau, f):=\int_M f(z)\theta_h(\tau, z),
\end{equation}
where $\theta_h(\tau, z)$ and $\Theta(\tau, z)$ are defined  in \S 3.2 of \cite{BF:traces}.

We have
\[I(\tau, f)\in H_{\frac32, \rho_L}.\]
To see this, note that the transformation properties follow from those of the theta function (\S 3.2 of \cite{BF:traces}).
The other conditions follow from the explicit description of the Fourier expansion of $I(\tau, f)$ 
given in Theorem~4.5 of \cite{BF:traces}. 

By (3.7) and (3.9) of \cite{BF:traces} and \eqref{eq:X-g} we have the relation
\[
	\theta_h(\tau,gz) = \theta_{g^{-1}h}(\tau,z)
\]
for any $g\in \Gamma_0^*(6)$.
From this and \eqref{thetalift-h} it follows that
\begin{equation} \label{eq:I-gh}
	I_{gh}(\tau,f) = I_h(\tau,f\big|_0 g).
\end{equation}

By Theorem~4.5 of \cite{BF:traces} we have 
\begin{equation}\label{bfthm}
I_h(\tau, f)=\sum_{m\geq 0}{\bf t}_f(h, m)q^m+\sum_{m>0} {\bf t}_f(h, -6m^2)q^{-6m^2}+N(\tau).
\end{equation}
A formula for the non-holomorphic part $N(\tau)$ is given in \cite{BF:traces} but we do not include it here.
Recall the definition \eqref{eq:gamma_def}. The terms of positive index $m$ in \eqref{bfthm} are given by 
\begin{equation}\label{posterms}
{\bf t}_f(h, m)
=\sum_{X\in \Gamma\backslash L_{h,m}} |\Gamma_X|^{-1} f(D_X),
\end{equation}
where 
$\Gamma_X\subseteq\Gamma$ is the stabilizer of $X$.
Let 
\[f\(\sigma_{\ell}\tau\)=\sum_{n\in \frac{1}{\alpha_{\ell}} \Z} a_{\ell}(n)q^{n}\]
denote the Fourier expansion of $f$ at the cusp $\ell$, where $\alpha_\ell$ denotes the width of $\ell$ as in Section~2.2.
By Proposition 4.7 of \cite{BF:traces}, the terms of negative index $-6m^2$ in \eqref{bfthm} are given by
\begin{equation}\label{negterms}
{\bf t}_f(h, -6m^2)= - \!\!\!\!\!\sum_{\ell\in \Gamma\backslash \operatorname{Iso}(V)} \sum_{n\in \frac{2m}{\beta_\ell}\N} a_\ell(-n) \left(v_\ell(h,-6m^2) e\left(\frac{r_{h,\ell} n}{2m}\right) +  v_\ell(-h,-6m^2) e\left(\frac{r_{-h,\ell}n}{2m}\right)\right),
\end{equation}
where $r_{h,\ell}$ is defined by 
\[
	\sigma_{\ell_X}^{-1}X=\pMatrix m {r_{h,\ell}} 0{-m} \ \text{ for any }X \in L_{h,-6m^2,\ell}.
\]
Note that ${\bf t}_f(h, -6m^2)=0$ for $m$ sufficiently large. 

\section{Proof of the algebraic formula}
The goal of this section is to prove Theorem~\ref{thm:spt-trace}.
To prepare for the proof,  recall the definition \eqref{eq:def-f} of 
$f\in M_0^!(\Gamma_0(6))$,
and  define $F(\tau)$ by 
\begin{equation}
\begin{aligned}
\label{sptmaassform}
F(\tau)&:=\sum_{n=1}^\infty\spt(n)q^{n-\frac1{24}}-\frac1{12}\cdot\frac{E_2(\tau)}{\eta(\tau)}+\frac{\sqrt{3i}}{2 \pi}
\int_{- \overline{\tau}}^{i \infty} \frac{\eta(w)}{\left(\tau+w \right)^{\frac{3}{2}}} dw\\
&=\sum_{n=0}^\infty s(n)q^{n-\frac1{24}}+\frac{\sqrt{3i}}{2 \pi}
\int_{- \overline{\tau}}^{i \infty} \frac{\eta(w)}{\left(\tau+w \right)^{\frac{3}{2}}} dw,
\end{aligned}
\end{equation} 
so that 
\begin{equation} \label{eq:def-s}
s(n) = \spt(n) + \frac{1}{12}(24n-1)p(n).
\end{equation}

Work of Bringmann \cite{Bringmann:duke} shows that $F(24\tau)$ is a  harmonic Maass form on $\Gamma_0(576)$ of weight $3/2$ and character $\leg{12}{\bullet}$, and that $F(24\tau)$ has eigenvalue $-1$
under the Fricke involution $W_{576}$.
Using these facts we find that
\begin{equation}\label{Ftrans}
F(\tau+1)=e\(-\tfrac1{24}\) F(\tau), \qquad  F(-1/\tau)=i^\frac12 \tau^\frac32 F(\tau).
\end{equation}
Now set
\begin{equation}\label{sptmaassdef}
\mathcal F(\tau):=\sum_{h\in L'/L}\(\tfrac{12}h\)F(\tau) {\mathfrak e}_h
\end{equation}
(we use the identification of $L'/L$ with $\Z/12\Z$ to define $\tleg{12}h$).
Using \eqref{Ftrans}, \eqref{weilrep}, 
and the fact that 
\[\sum_{h\in L'/L} \tleg{12}h e\tleg{hh'}{12}=\tleg{12}{h'}\sqrt{12},\]
we find that
\[\mathcal F(\tau)\in H_{\frac32,  \rho_L}.\]

\begin{proposition} \label{prop:I=F}
 We have $I(\tau, f)=24\, \mathcal F(\tau)$.
\end{proposition}	
\begin{proof}   
 For $d\mid 6$ we find that
\begin{equation}\label{wd}
f\big|_0 W_d=\begin{cases} 
	f& \text{ if } d=1,6,\\
	-f& \text{ if } d=2,3.
\end{cases}
\end{equation}
We first claim that 
for $h\in L'/L$ we have 
\begin{equation}\label{hclaim}
I_h\(\tau, f\)=\tleg{12}hI_1\(\tau, f\).
\end{equation}
When $(h, 12)=1$ the claim follows from \eqref{frickeaction}, \eqref{thetalift-h}, \eqref{eq:I-gh} and \eqref{wd}.
If $(h, 12)\neq 1$ then by \eqref{frickeaction}, $h$ is fixed by either $W_2$ or $W_3$.  
In this case,  \eqref{eq:I-gh} and \eqref{wd}
imply that $I_h(\tau, f)=0$, and \eqref{hclaim} follows.

Now let  
\[\mathcal G(\tau)\in H_{\frac32, \rho_L}\] 
denote the difference of the two forms in the statement of the lemma.
By \eqref{sptmaassdef} and \eqref{hclaim}, there is a  function $G$
such that
\[\mathcal G=\sum_{h\in L'/L}\tleg{12}hG \, \mathfrak e_h.\]
Arguing as above, we find that $G$ satisfies the transformation laws described by \eqref{Ftrans}.

Using \eqref{negterms}, we compute the principal part of $I_1(\tau,f)$ as follows. Since
\[
	\beta_{\ell_0} = \frac16, \quad \beta_{\ell_1} = 1, \quad \beta_{\ell_2} = \frac12, \quad  \text{ and } \quad \beta_{\ell_3} =\frac13 ,
\]
we see that ${\bf t}_f(1,-6m^2)=0$ for $m>\frac{1}{12}$. 
A computation shows that
\[
	v_\ell(1,-\tfrac1{24}) = 
	\begin{cases}
		1 & \text{ if } \ell=\ell_0, \\
		0 & \text{ otherwise,}
	\end{cases}
	\quad \text{ and } \quad
	v_\ell(-1,-\tfrac1{24}) = 
	\begin{cases}
		1 & \text{ if } \ell=\ell_1, \\
		0 & \text{ otherwise.}
	\end{cases}
\]
Furthermore, $r_{1,\ell_0} = r_{-1,\ell_1} = 0$, 
so ${\bf t}_f(1,-\frac1{24})=-2$. 
Therefore the principal part of $I_1(\tau, f)$ is given by $-2q^{-1/24}$, which agrees with the principal part of $24F(\tau)$.
It follows from \eqref{sptmaassdef} and \eqref{hclaim} that $\mathcal G(\tau)$ has trivial principal part.

Let $g=\xi_\frac32G$.  Then $g$ is holomorphic on $\H$ and at $\infty$, and we have 
\[
	g(\tau+1)=e\left(\tfrac1{24}\right)g(\tau) \quad \text{ and } \quad g(-1/\tau) = (-i)^\frac12\tau^\frac12 \, g(\tau).
\]
It follows that $g(\tau)$ is a constant multiple of   $\eta(\tau)$.
Theorem~3.6 of \cite{BF:geometric} then implies that $\mathcal G(\tau)$ is  a holomorphic modular form.
It follows that $\mathcal G=0$;
otherwise the product of $G$ with $\eta$ would be a non-zero modular form of weight $2$ for $\SL_2(\Z)$.
\end{proof}

\begin{remark}  A more direct approach to the proof of Proposition~\ref{prop:I=F} is to compute $N(\tau)$ using the formula in Theorem 4.5 of \cite{BF:traces}
and to match it directly to the non-holomorphic part of  \eqref{sptmaassform}.
\end{remark}

\begin{proof}[Proof of Theorem \ref{thm:spt-trace}]
Suppose that $n\geq 1$, and let $\widehat n:=n-\frac1{24}$.
By \eqref{bfthm}, \eqref{posterms}, \eqref{sptmaassform}, and Proposition~\ref{prop:I=F}, we have
\[
	s(n) = \frac1{24} \sum_{X\in \Gamma\backslash L_{1,\widehat n}} |\Gamma_X|^{-1} f(D_X).
\]
Note that for each $X$, we have $D_{-X}=D_X$, so we restrict our attention to the subset
\[
	L_{1,\widehat n}^+ := \left\{\pMatrix{b+\frac{1}{12}}{\frac{c}{6}}{-a}{-b-\frac{1}{12}} : a,b,c\in\Z, \ a>0, \text{ and } q(X) = \widehat n \right\}.
\]
There is a natural bijection between $L_{1,\widehat n}^+$ and $\sQ_{1-24n}^{(1)}$ given by
\[
	X=\pMatrix{b+\frac{1}{12}}{\frac{c}{6}}{-a}{-b-\frac{1}{12}} \longleftrightarrow Q_X := [6a, 12b+1, c].
\]
It is easy to check that the action of $\Gamma$ on $L_{1,\widehat n}$ translates under this bijection to the usual action of $\Gamma$ on $\sQ_{1-24n}^{(1)}$.
Since the stabilizer of $Q$ is trivial for every $Q\in \sQ_{1-24n}^{(1)}$, we have $|\Gamma_X|=1$ for all $X\in \Gamma\backslash L_{1,\widehat n}^+$.
A computation involving \eqref{eq:H-D-iso} shows that $D_X \mapsto \tau_{Q_X}$ under the isomorphism $\D\cong \H$. Thus we have
\[
	s(n) = \frac{1}{12}\sum_{Q\in \Gamma\backslash \sQ_{1-24n}^{(1)}} f(\tau_Q).
\]
The theorem follows from \eqref{algpart} and \eqref{eq:def-s}.
\end{proof}

\section{Poincar\'e series and the function $f(\tau)$}

In this section we construct the modular function $f(\tau)$ in terms of weak Maass-Poincar\'e series. To this end, we construct an auxiliary function $f(\tau,s)$, defined for $\re(s)>1$ and compute its Fourier expansion to obtain an analytic continuation of $f(\tau,s)$ to $\re(s)>\frac34$. 
We then show that  $f(\tau,1)=f(\tau)$.

Recall \eqref{eq:gamma_def} and 
  write 
  \[\tau=x+iy,\qquad s=\sigma+it.\]
 Letting $\Gamma_\infty=\{\pmatrix 1*01\}$ denote the stabilizer of $\infty$, we define
\[
	F(\tau,s) := \sum_{\gamma\in \Gamma_\infty\backslash\Gamma} \phi_s(\im\gamma\tau)e(-\re\gamma\tau),
\]
where
\begin{equation} \label{eq:def-phi}
	\phi_s(y) := 2\pi\sqrt{y} \, I_{s-\frac12}(2\pi y).
\end{equation}
Since $\phi_s(y) \ll y^\sigma$ as $y\to 0$,
we have
\[
	F(\tau,s) \ll y^\sigma \sum_{\pmatrix **cd\in\Gamma_\infty\backslash\Gamma} |c\tau+d|^{-2\sigma},
\]
so $F(\tau,s)$ converges normally for $\sigma>1$. 
A computation involving (13.14.1) of \cite{NIST:DLMF} shows that
\begin{equation} \label{eq:delta-0-F}
	\Delta_0 F(\tau,s) = s(1-s)F(\tau,s).
\end{equation}

Let $\mu$ denote the M\"obius function, and define
\[
	f(\tau,s) := \sum_{r\mid 6} \mu(r) F(W_r \tau,s).
\]
The following proposition gives the Fourier expansion of $f(\tau,s)$ in terms of
the ordinary Kloosterman sum
\begin{equation} \label{eq:ordinary-kloo}
	k(m,n;c) := \sum_{\substack{d\bmod c \\ (c,d)=1}} e\left(\frac{m\bar d+nd}{c}\right)
\end{equation}
and the $I$, $J$, and $K$-Bessel functions
(here $\bar d$ is the multiplicative inverse of $d$ modulo $c$).

\begin{proposition} \label{prop:f-tau-s-expansion}
For $\sigma>1$ we have
\[
	f(\tau,s) = 2\pi\sqrt{y}\,I_{s-\frac12}(2\pi y)e(-x) + a_s(0)y^{1-s} + 2\sqrt{y}\,\sum_{n\neq 0}a_s(n) K_{s-\frac12}(2\pi|n|y)e(nx),
\]
where
\begin{align*}
	a_s(0) &= \frac{2\pi^{s+1}}{(s-\frac12)\Gamma(s)}\sum_{r\mid 6}\mu(r)\sum_{\substack{0<c\equiv 0(6/r) \\ (c,r)=1}} \frac{k(-\bar r,0;c)}{(c\sqrt{r})^{2s}},
	\\
	a_s(n) &=2\pi\sum_{r\mid 6}\mu(r)\sum_{\substack{0<c\equiv 0(6/r) \\ (c,r)=1}} \frac{k(-\bar r,n;c)}{c\sqrt{r}} \times
	\begin{dcases}
		I_{2s-1}\pfrac{4\pi\sqrt{n}}{c\sqrt{r}} & \text{ if }n>0, \\
		J_{2s-1}\pfrac{4\pi\sqrt{|n|}}{c\sqrt{r}} & \text{ if }n<0.
	\end{dcases}
\end{align*}
\end{proposition}

\begin{proof}
The function $f(\tau,s)-\phi_s(y)e(-x)$ has at most polynomial growth as $y\to\infty$.
For $r\mid 6$, $r\neq 1$, 
a complete set of representatives for $\Gamma_\infty \backslash \Gamma W_r$ is given by
\[
	\left\{ \pMatrix{ra}{*}{rc}{rd} : \ c>0, \ \gcd(c,rd)=1, \ (6/r) \mid c,\  a\in\{1, \dots, c-1\},\  a\equiv\bar{rd}\pmod c \right\}.
\]
So
we have the Fourier expansion
\[
	f(\tau,s) = \phi_s(y)e(-x) + \sum_{n\in \Z} \sum_{r\mid 6} \mu(r) A_r(n,y,s) e(nx),
\]
where
\begin{align*}
	A_r(n,y,s)
	&= \sum_{\substack{\gamma \in \Gamma_\infty \backslash \Gamma W_r \\ c(\gamma)>0}} \int_0^1 \phi_s(\im\gamma \tau) e(-\re\gamma\tau) e(-nx) \, dx.
\end{align*}

For $\gamma=\pmatrix{ra}{b}{rc}{rd}\in \Gamma_\infty \backslash \Gamma W_r$ with $c(\gamma)>0$ we write
\[
	\gamma\tau=\frac{ra \, \tau+b}{rc \, \tau+rd} = \frac{a}{c} - \frac{1}{rc^2(\tau+d/c)}
\]
and  make the change of variable $x\to x-d/c$ to obtain
\begin{align*}
	A_r(n,y,s) &= \sum_{\substack{0<c\equiv 0(6/r) \\ (c,rd)=1}} e\left(\frac{-a+nd}{c}\right) \int_{d/c}^{1+d/c} \phi_s\left(\frac{y}{rc^2|\tau|^2}\right) e\left(\frac{x}{rc^2|\tau|^2}-nx\right) \, dx.
\end{align*}
Since $a\equiv \bar{rd}\pmod{c}$ we have
\[
	\sum_{\substack{d\bmod c \\ (d,c)=1}} e\pfrac{-a+nd}{c} = k(-\bar{r},n;c),
\]
so that
\begin{align*}
	A_r(n,y,s) &=2\pi\sum_{\substack{0<c\equiv 0(6/r) \\ (r,c)=1}} \frac{k(-\bar r,n;c)}{c\sqrt{r}} y^{\frac12}\int_{-\infty}^\infty |\tau|^{-1} \, I_{s-\frac12}\left(\frac{2\pi y}{rc^2|\tau|^2}\right) e\left(\frac{x}{rc^2|\tau|^2} - nx\right) \,dx.
\end{align*}
Let $I$ denote the integral above. 
We make the substitution $x=yu$ and set $A=\frac{1}{rc^2y}$ and $B=-ny$, so that
\[
	I=\int_{-\infty}^\infty (u^2+1)^{-\frac12} I_{s-\frac12}\pfrac{2\pi A}{u^2+1} e\left(\frac{Au}{u^2+1} + Bu\right) \, du.
\]
Using Lemma 5.5 on page 357 and (xiv) and (xv) on page 345 of \cite{Hejhal}, and the fact that
\[
	2^{1-2s}\sqrt{\pi}\, \frac{\Gamma(2s)}{\Gamma(s+\frac12)\Gamma(s)} = 1,
\] 
we find that
\[
	I=
	\begin{dcases}
		2 K_{s-\frac12}(2\pi B) J_{2s-1}\left(4\pi\sqrt{AB}\right) &\text{ if }B>0, \\
		\frac{\pi^s A^{s-\frac12}}{(s-\frac12)\Gamma(s)} &\text{ if }B=0, \\
		2 K_{s-\frac12}(2\pi |B|) I_{2s-1}\left(4\pi\sqrt{A|B|}\right) &\text{ if }B<0.
	\end{dcases}
\]
The proposition follows.
\end{proof}

The function $f(\tau,s)$ has an analytic continuation to $\sigma>\frac34$, as we now show.
Suppose that $\frac34<\sigma_0<\frac32$, and fix $\epsilon_0$ with $0<\epsilon_0< 2\sigma_0-\frac32$.
We will show that the Fourier expansion in Proposition~5 converges absolutely and  uniformly for 
$s$ in the region $R$ defined by $\sigma_0\leq \sigma\leq \frac32$, $|t|\leq T$
(the estimates below are for $s\in R$).
Using the Weil bound  \cite{Weil:exponential} for Kloosterman sums we have
\[
	k(a,b;c) \ll \gcd(a,b,c)^{\frac12} c^{\frac12+\ep_0}, 
\]
from which 
\[
	a_s(0) \ll \sum_{c>0} c^{-2\sigma_0+\frac12+\ep_0} \ll 1.
\]
By (10.40.2) of \cite{NIST:DLMF} we have 
\begin{equation}\label{KBesselsize}
	\sqrt{y} \, K_{s-\frac12}(2\pi|n|y) \ll |n|^{-\frac12}e^{-2\pi|n|y} \quad \text{as $n\to\infty$}.
\end{equation}
From  (10.40.1) and (10.30.1) of \cite{NIST:DLMF} we have
\begin{gather*}
	I_{2s-1}(x) \ll \frac{e^x}{\sqrt{x}} \quad \text{as }x\to\infty, \\
	I_{2s-1}(x) \ll x^{2s-1} \quad \text{as }x\to 0.
\end{gather*}
Suppose that $n>0$. Taking absolute values in the series defining $a_s(n)$, we find that
\begin{equation}\label{assize}
a_s(n)  \ll   n^{-\frac14}\sum_{c<\sqrt n} c^{\ep_0} \, e^{\frac{4\pi\sqrt{n}}{c}} + n \sum_{c\geq \sqrt n}  c^{-2\sigma_0+\frac12+\ep_0} 
 \ll e^{4\pi\sqrt n}n^{\frac14+\ep_0}+n \ll e^{6\pi\sqrt n}.
\end{equation}
From (10.7.8) and (10.7.3) of \cite{NIST:DLMF} we have
\begin{gather*}
	J_{2s-1}(x) \ll \frac{1}{\sqrt x} \quad \text{as }x\to\infty, \\
	J_{2s-1}(x) \ll x^{2s-1} \quad \text{as }x\to 0.
\end{gather*}
Arguing as above we find that for $n<0$ we have $a_s(n)\ll n$.
With \eqref{assize} and  \eqref{KBesselsize}, this shows that the Fourier expansion converges absolutely and uniformly for $s\in R$.
This provides the analytic continuation of  $f(\tau,s)$ to $\sigma>\frac34$.

Since 
\[
	2\sqrt{y}\,K_{\frac12}(2\pi |n|y)= |n|^{-\frac12}e^{-2\pi |n|y} \quad \text{ and } \quad 2\pi\sqrt{y}\,I_{\frac12}(2\pi y)=2\sinh(2\pi y),
\]
the Fourier expansion of $f(\tau,1)$ is
\begin{equation}\label{fourierftau}
	f(\tau,1) = e(-\tau) + a_1(0) + \sum_{n>0}\frac{a_1(n)}{\sqrt{n}}e(n\tau) - e(-\bar{\tau}) + \sum_{n<0}\frac{a_1(n)}{\sqrt{|n|}}e(n\bar{\tau}).
\end{equation}
Using \eqref{eq:delta-0-F} and \eqref{fourierftau} we find that $\xi_0 f(\tau,1)$ is 
a cusp form of weight $2$ on $\Gamma_0(6)$, so it equals $0$.
Therefore $f(\tau,1)$ is holomorphic on $\H$. Since the principal parts of $f(\tau,1)$ and $f(\tau)$ are equal, we conclude that
\begin{equation} \label{eq:f-tau-1=f-tau}
	f(\tau,1) = f(\tau),
\end{equation}
as desired.

\section{Proof of Theorem \ref{thm:spt-exact}}

By Theorem \ref{thm:spt-trace} and equations \eqref{eq:f-tau-1=f-tau}, \eqref{eq:def-s}, and \eqref{eq:root-compatible}  we have
\begin{align*}
	s(n) &= \frac{1}{12}\sum_{Q\in \Gamma\backslash \sQ_{1-24n}^{(1)}} f(\tau_Q) \\
	&= \lim_{s\to 1^+} \frac{1}{12} \sum_{Q\in \Gamma\backslash \sQ_{1-24n}^{(1)}} \sum_{d\mid 6} \sum_{\gamma \in \Gamma_\infty \backslash \Gamma} \mu(d) \phi_s(\im \tau_{\gamma W_d Q})e(-\re\tau_{\gamma W_d Q}).
\end{align*}
By \eqref{eq:W-d-r-r'} and \eqref{eq:Qn-decomp-W-d} the map $\left(\gamma,d,Q\right) \mapsto \gamma W_d\,Q$ is a bijection 
\begin{equation} \label{eq:bij-neg-disc}
	\Gamma_\infty\backslash\Gamma \times \{1,2,3,6\} \times \Gamma\backslash \sQ^{(1)}_{1-24n} \longleftrightarrow \Gamma_\infty \backslash \sQ_{1-24n}.
\end{equation}
If $Q\in \sQ_{1-24n}^{(1)}$ and $Q'=W_dQ=[a,b,*]$ then $\mu(d)=\pfrac{12}b$ by \eqref{eq:W-d-r-r'} and \eqref{eq:r'-r-cases}. Thus we have
\begin{align*}
	s(n) = \lim_{s\to 1^+} \frac{1}{12} \sum_{\substack{Q\in\Gamma_\infty \backslash \sQ_{1-24n} \\ Q=[a,b,*]}} \pfrac{12}{b} \phi_s\left( \frac{\sqrt{24n-1}}{2a}\right)e\left(\frac{b}{2a}\right).
\end{align*}
Since $\pmatrix 1k01 [a,b,*]=[a,b-2ka,*]$, there is a bijection
\begin{equation*} \label{eq:sQ-infty-bijection}
	\Gamma_\infty\backslash \sQ_{1-24n} \longleftrightarrow \left\{(a,b) : a>0, \  \ 6\mid a, \ \  0\leq b<2a, \ \ b^2\equiv1-24n\pmod{4a}\right\},
\end{equation*}
which, together with \eqref{eq:def-phi}, gives
\begin{align*}
	s(n) &= \lim_{s\to 1^+} \frac{\pi}{6\sqrt{2}} (24n-1)^\frac14 \sum_{6\mid a>0} a^{-\frac12} I_{s-\frac12} \left( \frac{\pi\sqrt{24n-1}}{a} \right) \sum_{\substack{b\bmod{2a} \\ b^2\equiv 1-24n (4a)}} \pfrac{12}{b} e\pfrac{b}{2a}.
\end{align*}
Writing $a=6c$, we see that the inner sum is equal to
\[
	\frac{1}{2} \sum_{\substack{b\bmod 24c \\ b^2\equiv 1-24n(24c)}} \pfrac{12}{b} e\pfrac{-b}{12c}.
\]
By Proposition 6 of \cite{Andersen:singular} (see also \cite{Whiteman}), this equals
\[
	\frac{2\sqrt{3}}{\sqrt{c}} A_c(n).
\]
We conclude that
\begin{equation} \label{eq:s-exact-limit}
	s(n) = \lim_{s\to 1^+} \frac{\pi}{6}(24n-1)^\frac14 \sum_{c=1}^\infty \frac{A_c(n)}{c} I_{s-\frac12} \left( \frac{\pi\sqrt{24n-1}}{6c} \right).
\end{equation}
To finish the proof of Theorem~ \ref{thm:spt-exact} we need 
to interchange  the limit and the sum in  \eqref{eq:s-exact-limit}.
To justify this, we apply partial summation and
Theorem~\ref{Athm}.

Set $a:=\frac{\pi\sqrt{24n-1}}{6}$, suppose that $s\in [1,2]$, and define   
\begin{equation}\label{bfadef}
{\bf A}_n(x):=\sum_{c\leq x} \frac{A_c(n)}{c}.
\end{equation}
By Theorem~\ref{Athm} we have ${\bf A}_n(x)\ll_{\epsilon, n}x^{\frac16+\ep}$ for any $\ep>0$.
Partial summation, together with \eqref{festimate} and  Lemma \ref{lem:bessel} below, gives
\[\begin{aligned}
\sum_{c>N} \frac{A_c(n)}{c} I_{s-\frac12}\pfrac ac&=\lim_{x\to\infty}{\bf A}_n(x)I_{s-\frac12}\pfrac ax-{\bf A}_n(N)I_{s-\frac12}\pfrac aN-\int_N^\infty {\bf A}_n(t)
\(I_{s-\frac12}(a/t)\)'\, dt\\
&\ll_aN^{\frac23-s+\epsilon}+\int_N^\infty t^{-\frac13-s+\epsilon}\, dt
\ll_aN^{-\frac13+\epsilon}.
\end{aligned}\]
It follows that the series \eqref{eq:s-exact-limit} converges uniformly for $s\in [1,2]$. Interchanging 
the limit and the sum gives Theorem \ref{thm:spt-exact}.
It remains to prove the following straightforward lemma.

\begin{lemma}  \label{lem:bessel}
Suppose that 
$a>0$ is fixed and that $\frac12\leq\nu\leq \frac32$.
Then 
\[\left|\(I_\nu(a/x)\)'\right| \ll_{a}{x^{-\nu-1}} \ \ \ \text{as $x\to\infty$}.\]
\end{lemma}
\begin{proof}

By (10.29.1) of \cite{NIST:DLMF} we find that 
\[
	I_\nu'(x)=\tfrac12\left(I_{\nu-1}(x)+I_{\nu+1}(x)\right).
\]
For fixed $x$, the function $I_\nu(x)$ is decreasing as a function of $\nu$.  
Therefore
\begin{equation}\label{monotonebessel}
\left|\(I_\nu(a/x)\)'\right| =\frac{a}{2x^2}\left(I_{\nu-1}(a/x)+I_{\nu+1}(a/x)\right)\leq\frac{a}{x^2}I_{\nu-1}(a/x).
\end{equation}
From (10.30.1) of \cite{NIST:DLMF} we have 
\begin{equation}\label{festimate}
I_\nu(a/x)\ll_a x^{-\nu}\ \ \text{as $x\to\infty$ \ \ for \ \ $\nu\in [1/2, 3/2]$}.
\end{equation}
The lemma follows from \eqref{monotonebessel} and \eqref{festimate}.
\end{proof}

\section{Sums of Kloosterman sums}\label{kloostersection}
In this section we will consider  sums of  Kloosterman sums $S(m, n, c,\chi)$ associated to a multiplier in weight $k$, which were
 studied when $m, n>0$ by Goldfeld and Sarnak \cite{GS:kloosterman}.
 Work of  Folsom-Ono \cite{FO:duality} and  Pribitkin \cite{Pribitkin:kloosterman} applies to the case of general  $m$ and $n$.
 For completeness we record a general asympotic formula here,
referring the reader to \cite{GS:kloosterman} for details. 

Let  $\Gamma$ be a finite-index subgroup of $\SL_2(\Z)$ which contains $-I$. Suppose that $k\in \R$ and   that $\chi$ is a multiplier on $\Gamma$ for  the weight $k$.
Suppose that $q$ is the smallest positive integer
for which $\pmatrix 1q01\in \Gamma$
and define $\alpha\in [0, 1)$ by 
\[\chi\left(\!\pMatrix1q01\!\right)=e(-\alpha).\]
For simplicity, we write
\[\widetilde n:=\frac{n-\alpha}q.\]
For $c>0$,  the generalized Kloosterman sum is given by 
\[S(m, n, c,\chi):=\sum_{\substack{0\leq\, a,d\,<qc \\ \gamma=\pmatrix abcd\in\Gamma}}\overline{\chi(\gamma)} \, e\(\frac{\widetilde m a+\widetilde n d}{c}\), \]
and Selberg's Kloosterman zeta function is defined as
\[
	Z_{m,n}(s,\chi) := \sum_{c>0}\frac{S(m,n,c,\chi)}{c^{2s}}.
\]

The space 
$L^2\(\Gamma\backslash \H, \chi, k\)$ consists of functions $f: \H\to\C$ such that
\[f\(\gamma \tau\)=\chi(\gamma)\(\frac{c\tau+d}{|c\tau+d|}\)^kf(\tau) \quad \text{ for all } \gamma\in \Gamma\]
and $\|f\|<\infty$, where
\[
	\|f\|^2 := \iint_{\Gamma\backslash\H}|f(\tau)|^2\,\frac{dxdy}{y^2}.
\]
The operator
\[\widetilde \Delta_{k}:=y^2\(\frac{\partial^2}{\partial x^2}+\frac{\partial^2}{\partial y^2}\)-iky\frac\partial{\partial x}\]
(which is not  the operator $\Delta_k$ in \S \ref{thetaliftsec}) 
has a self-adjoint extension to $L^2\(\Gamma\backslash \H, \chi, k\)$ with real spectrum.
The asymptotic formula of \cite{GS:kloosterman} depends on the discrete spectrum, which 
we denote by 
\[\lambda_0\leq\lambda_1\leq  \lambda_2\leq\dots.\]
For each $j$  with $\lambda_j<\frac14$ let $u_j$ be 
the normalized Maass cusp form corresponding to $\lambda_j$
and define 
$s_j\in \(\frac12, 1\)$ by
\[\lambda_j=s_j(1-s_j).\]
We have the expansion
\begin{equation}\label{uj}
u_j(\tau)=\sum_{m=-\infty}^\infty
\widehat{u}_j(m,y) e(\widetilde m x),
\end{equation}
where
\begin{equation} \label{ujhat}
	\widehat{u}_j(m,y) = 
	\begin{dcases}
		\rho_j(m) W_{\frac k2 \sgn(\widetilde m),s_j-\frac 12}(4\pi|\widetilde m|y) e(\widetilde m x) & \text{ if } \widetilde m\neq 0, \\
		\rho_j(0) y^{s_j} + \rho'_j(0) y^{1-s_j} &\text{ if } \widetilde m=0.
	\end{dcases}
\end{equation}
Define
\[
	\beta := \limsup_{c\to\infty} \frac{\log|S(m,n,c,\chi)|}{\log c}.
\]
With this notation we have the following
\begin{proposition}\label{gsplus}
Suppose that $m>0$ and that $n\in \Z$.  For any $\epsilon>0$ we have
\begin{equation*}
	\sum_{c\leq x} \frac{S(m,n,c,\chi)}{c}=\sum_{\frac12<s_j<1}\tau_j(m,n)\frac{x^{2s_j-1}}{2s_j-1}+O\(x^{\frac{\beta}3+\epsilon}\), 
	\end{equation*}
where  the sum runs over those $j$ with $\lambda_j<\frac14$ as above and 
\begin{equation}\label{taudef}
\tau_j(m, n)=2q^2i^k\,\overline{\rho_j(m)}\rho_j(n)\pi^{1-2s_j}(4\widetilde m|\widetilde n|)^{1-s_j}
\frac{\Gamma\(s_j+\sgn(\widetilde n)\frac k2\)\Gamma\(2s_j-1\)}{\Gamma\(s_j-\frac k2\)}.
\end{equation}
The implied constant depends on $k$, $\chi$, $m$, $n$, $\epsilon$, and $\Gamma$.
\end{proposition}
When $n>0$ this is Theorem~2 of \cite{GS:kloosterman}, but the constants  in (3.2) of \cite{GS:kloosterman} differ  from those in \eqref{taudef}.
Figure \ref{kloosterfig} in Section~\ref{kloostereta} gives an example which supports the accuracy of  \eqref{taudef}.
The existence of such a formula is implicit in  \cite{FO:duality}.

For the case when $n\leq 0$ 
we argue as in Lemma~2 of \cite{GS:kloosterman}, relating $Z_{m,n}(s,\chi)$ to the inner product of two Poincar\'e series. We compute
\begin{equation*}
\langle P_m\(\tau, s, \chi, k\), \overline{P_{1-n}\(\tau, s+2, \overline{\chi}, -k\)}\rangle=
\frac{(-i)^k}{4^{s+1}\pi\widetilde n^2}\cdot
\frac{\Gamma\(2s+1\)}{\Gamma\(s-\frac k2\)\Gamma\(s+\frac k2+2\)}Z_{m, n}(s, \chi)+R(s),
\end{equation*}
where $R(s)$ is holomorphic in $\sigma>\frac12$ and is $O\(\frac1{\sigma-\frac12}\)$ in this region. Arguing as in Section~2 of \cite{GS:kloosterman} we  find that for $m>0$ and for all $n$  we have  
\[\Res_{s=s_j}Z_{m, n}(s, \chi)=q^2i^k\overline{\rho_j(m)}\rho_j(n)\pi^{1-2s_j}(4\widetilde m |\widetilde n|)^{1-s_j}
\frac{\Gamma\(s_j+\sgn(\widetilde n)\frac k2\)\Gamma\(2s_j-1\)}{\Gamma\(s_j-\frac k2\)}.\]
The proposition follows by the method of Section 3 of \cite{GS:kloosterman}.

\section{Sums of Kloosterman sums for the $\eta$-multiplier}\label{kloostereta}

We specialize the results of Section \ref{kloostersection} to $k=\frac 12$,  $\Gamma=\SL_2(\Z)$, and  $\chi$ the multiplier system attached to the $\eta$-function.
For matrices in $\SL_2(\Z)$ with $c>0$ we have (see, for example, \S 2.8 of \cite{Iwaniec:classical})
\begin{equation}\label{chidef}
	\chi\left(\!\pMatrix abcd\!\right) = \sqrt{-i} \, e\left(\frac{a+d}{24c}\right)e^{-\pi i s(d,c)},
\end{equation}
where $s(d,c)$ is the Dedekind sum defined in \eqref{eq:def-dedekind-sum}.
In this case we have $q=1$ and $\alpha=\frac{23}{24}$. 
Recall that the pentagonal numbers are those numbers  of the form $\frac{k(3k\pm 1)}2$ for $k\in \Z$.
We have the following (c.f. \cite[Theorem 4.5]{Sarnak:additive}).
\begin{theorem} \label{thm:A-bound}
Suppose that $m>0$ and that $n\in \Z$.  If  $m-1=\frac{k_1(3k_1\pm 1)}2$ and $n-1=\frac{k_2(3k_2\pm 1)}2$ are both pentagonal then for any $\epsilon>0$ we have 
\[
	\sum_{c\leq x} \frac{S(m,n,c,\chi)}{c} = \sqrt{i}\, (-1)^{k_1+k_2} \, \mfrac{12\sqrt{3}}{\pi^2} \, x^{\frac12}+O\(x^{\frac16+\epsilon}\).
\]
Otherwise we have
\[
	\sum_{c\leq x} \frac{S(m,n,c,\chi)}{c} = O\(x^{\frac16+\epsilon}\).
\]
The implied constants depend on $m$,  $n$, and $\ep$.
\end{theorem}
Recalling the definition \eqref{Acn}, we find that
\[
	 \sqrt{i} \, A_c(n)=S\(1,-n+1,c,\chi\),
\]
so Theorem~\ref{Athm} is an immediate corollary.
Figure \ref{kloosterfig} shows values of the 
summatory function ${\bf A}_n(x)$ for the  pentagonal number $-n=1$  (along with the  asymptotic curve) and the non-pentagonal number $-n=-1$. 

\renewcommand{\thesubfigure}{\Alph{subfigure}}

\begin{figure}[h]
\centering
\begin{subfigure}[t]{0.48\textwidth}
	\includegraphics[width=\textwidth]{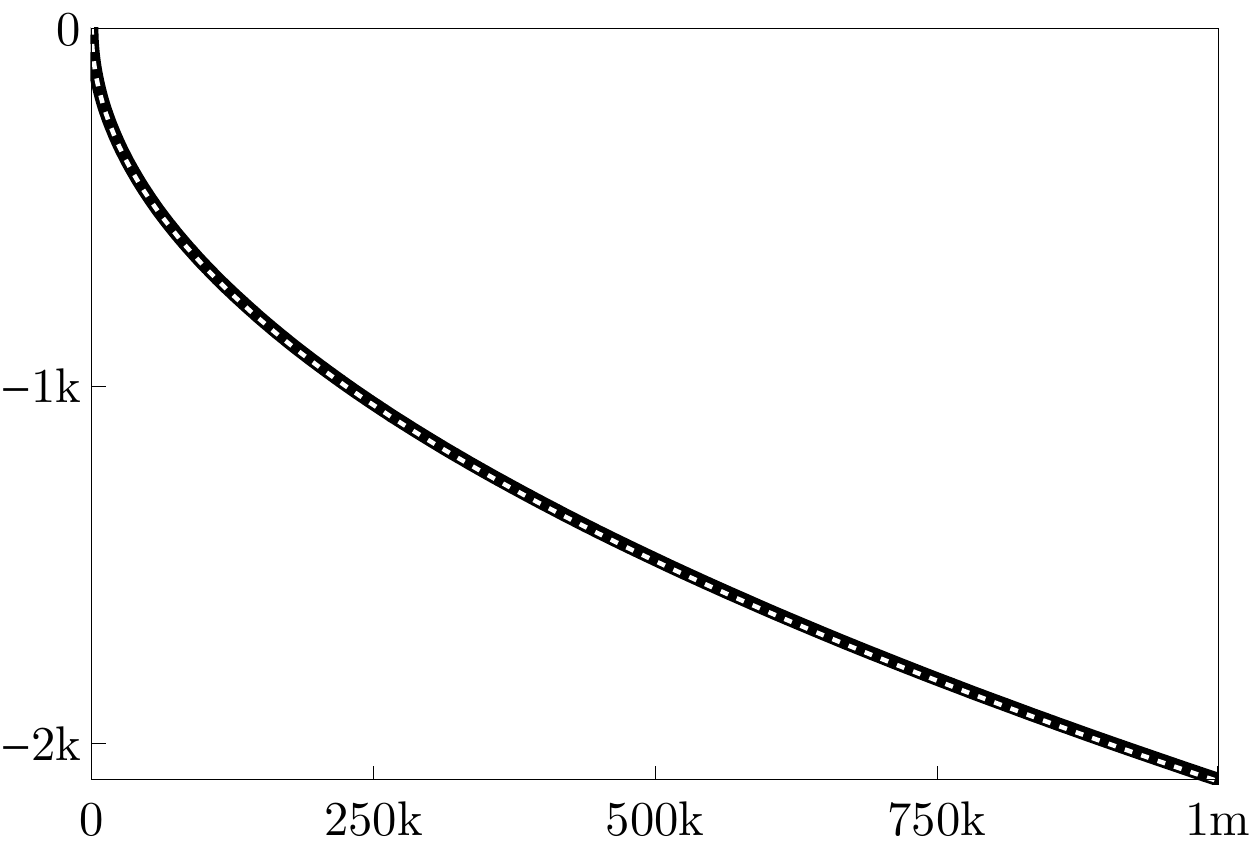}
    \caption{\parbox[t]{.65\textwidth}{${\bf A}_{-1}(x)$ (solid black line) vs. \\  $-\frac{12\sqrt{3}}{\pi^2}\sqrt{x}$ (dotted white line)}}
\end{subfigure}
\hfill
\begin{subfigure}[t]{0.48\textwidth}
	\includegraphics[width=\textwidth]{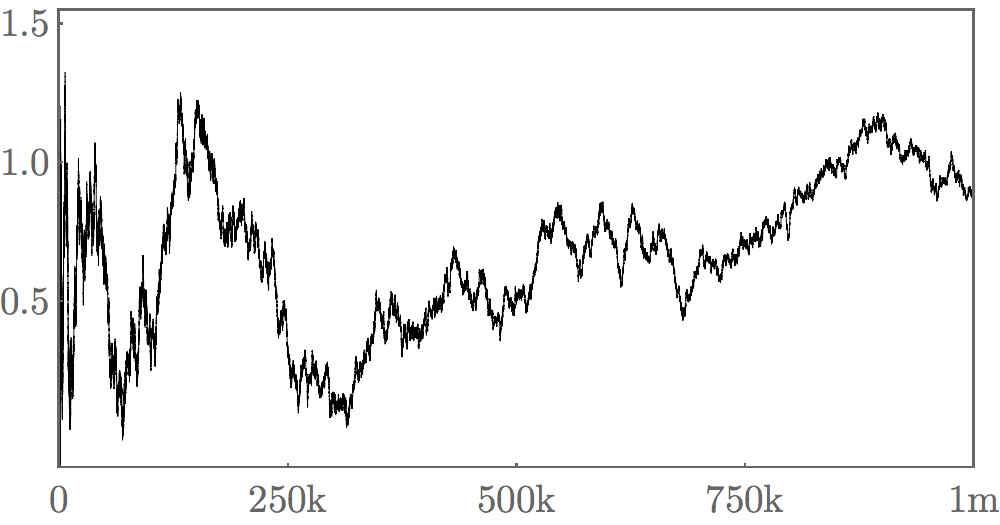}
    \caption{${\bf A}_{1}(x)$}
    \label{kfigB}
\end{subfigure}
\caption{Plots of ${\bf A}_n(x) = \sum\limits_{c\leq x}\frac{A_c(n)}{c}$ for $n=\pm 1$.}
\label{kloosterfig}	
\end{figure}

In the proof we will need to know the Petersson norm of the eta function,
which is given by the next lemma.
\begin{lemma}\label{lem:eta-norm} We have $\|y^{\frac 14}\eta\|^2 = \frac{\pi}{3\sqrt{6}}$.
\end{lemma}

\begin{proof}
For $\re(s)>1$, let \[E(\tau,s) = \sum_{\gamma \in \Gamma_\infty \backslash \SL_2(\Z)} (\im \gamma \tau)^s\] denote the nonholomorphic Eisenstein series (see, for example, \cite{Zagier:eisenstein}), and define
\[
	I(s) := \int_{\SL_2(\Z)\backslash \H} E(\tau,s) y^\frac12 |\eta(\tau)|^2 \, \frac{dxdy}{y^2}.
\]
Since $E(\tau,s)$ has a pole at $s=1$ with residue $3/\pi$, we have 
\[
	\|y^{\frac 14}\eta\|^2 = \Res_{s=1} I(s).
\]
On the other hand, we have
\begin{align*}
	I(s) &= \int_{\Gamma_\infty \backslash \H} y^{s+\frac12} |\eta(\tau)|^2 \, \frac{dxdy}{y^2} \\
	&= \sum_{n,m\geq 1} \pfrac{12}{nm} \int_0^\infty y^{s-\frac12} e^{-2\pi (n^2+m^2)\frac y{24}}\, \frac{dy}{y} \times \int_0^1 e\left(\frac{n^2-m^2}{24}x\right) \, dx \\
	&= \ptfrac{6}{\pi}^{s-\frac12} \Gamma(s-\tfrac12) \sum_{\substack{n\geq 1 \\ (n,6)=1}} \frac{1}{n^{2s-1}} 
	= \ptfrac{6}{\pi}^{s-\frac12} \Gamma(s-\tfrac12) (1-\tfrac12)(1-\tfrac13)\zeta(2s-1).
\end{align*}
From this we obtain
\[
	\Res_{s=1} I(s) = \frac{1}{\sqrt{6}},
\]
as desired.
\end{proof}

\begin{proof}[Proof of Theorem \ref{thm:A-bound}]
By Proposition~1.2 of \cite{Sarnak:additive} and the discussion that follows, 
we see that the minimal eigenvalue of $\widetilde\Delta_\frac12$
is $\lambda_0=\frac3{16}$, so that $s_0=\frac34$. This is achieved by  a unique normalized Maass cusp form
in
$L^2\(\SL_2(\Z)\backslash \H, \chi, \tfrac12\)$,
namely
\begin{equation} \label{eq:u0-eta}
u_0(\tau)= \frac{y^\frac14\eta(\tau)}{\|y^\frac14\eta\|} = \sqrt{\mfrac 3\pi} \, (6y)^{\frac14} \eta(\tau).
\end{equation}

Bruggeman  studied families of  modular forms on $\SL_2(\Z)$ parametrized by their weight. His work \cite[Theorem 2.15]{Bruggeman} shows that there are no exceptional eigenvalues in this case; in other words we
have $\lambda_1>\frac14$.   In forthcoming work \cite{AhlgrenAndersen} of the present authors 
we introduce a theta lift which gives a Shimura-type correspondence between the space in question and  a space of weight
$0$ Maass cusp forms of level 6.   
This, together with computations of Booker and Str\"ombergsson as in \cite[Section 4]{booker-strombergsson} gives the  lower bound 
$\lambda_1>3.86$.

By Theorem 3 of \cite{Andersen:singular}
we can take $\beta=1/2$.
Proposition \ref{gsplus} and the bounds on $\lambda_1$ imply that
\[
	\sum_{c\leq x} \frac{S(m,n,c,\chi)}{c}  = 2\tau_0(m,n)x^{\frac12} + O(x^{\frac16+\epsilon}),
\]
where (with $\widetilde n=n-\frac{23}{24}$ as before) we have
\[
	\tau_0(m,n) = 2\sqrt{2i} \, \pi^{-\frac12} \, \bar{\rho_0(m)}\rho_0(n) \widetilde m^\frac14 |\widetilde n|^{\frac 14} \Gamma\left(\tfrac 34+\tfrac 14\sgn(\widetilde n)\right).
\]
Equations \eqref{uj}, \eqref{ujhat}, and \eqref{eq:u0-eta} give the relation
\begin{equation*}
	\sum_{m\in\Z} \rho_0(m) W_{\frac14 \sgn\( \widetilde m\), \frac14}\left(4\pi\left| \widetilde m\right|y\right) e\left( \widetilde m x\right) 
	=\sqrt{\mfrac 3\pi}\,(6y)^\frac14 q^{\frac1{24}}\(1+\sum_{k=1}^\infty (-1)^k\(q^{\frac{k(3k-1)}2}+q^{\frac{k(3k+1)}2}\)\).
\end{equation*}
Since
\[
	W_{\frac14,\frac14}(y) = y^\frac14 e^{-y/2},
\]
we find that
\[
	\rho_0(m) = 
	\begin{dcases}
		(-1)^k \sqrt{\mfrac 3\pi} \,6^{\frac14} (4\pi \widetilde m)^{-\frac14} &\text{ if } m-1=\tfrac{k(3k\pm 1)}{2}, \\
		0 & \text{ otherwise.}
	\end{dcases}
\]
Therefore
\[
	\tau_0(m,n) = 
	\begin{dcases}
		(-1)^{k_1+k_2}\frac{6\sqrt{3i}}{\pi^2} & \text{ if } m-1=\tfrac{k_1(3k_1\pm1)}{2} \text{ and } n-1=\tfrac{k_2(3k_2\pm1)}{2}, \\
		0 &\text{ otherwise,}
	\end{dcases}
\]
and the theorem follows.
\end{proof}

\bibliographystyle{plain-i}
\bibliography{spt-bib}

\end{document}